\documentclass[submission,copyright,creativecommons]{eptcs}
\providecommand{\event}{ACT 2024} 

\usepackage{iftex}

\ifpdf
  \usepackage{underscore}         
  \usepackage[T1]{fontenc}        
\else
  \usepackage{breakurl}           
\fi

\title{Combs, Causality and Contractions in Atomic Markov Categories}
\author{Dario Stein \quad\quad M\'ark Sz\'eles
\institute{Radboud University Nijmegen, The Netherlands}
\email{\{dario.stein,mark.szeles\}@ru.nl}
}

\newcommand{\titlerunning}{Combs, Causality and Contractions}
\newcommand{\authorrunning}{D. Stein \& M. Sz\'eles}

\hypersetup{
  bookmarksnumbered,
  pdftitle    = {\titlerunning},
  pdfauthor   = {\authorrunning},
  pdfsubject  = {EPTCS},               
  pdfkeywords = {markov categories, categorical probability, traces, combs, causality}
}

\usepackage{amsmath,amssymb,amsfonts,stmaryrd,amsthm}
\usepackage{quiver}
\usepackage{hyperref}
\usepackage{tikz-cd}

\usepackage[export]{adjustbox}

\usepackage{tikzit}

\tikzstyle{morphism}=[fill=white, draw=black, shape=rectangle]
\tikzstyle{medium box}=[fill=white, draw=black, shape=rectangle, minimum width=0.8cm, minimum height=0.9cm]
\tikzstyle{large morphism}=[fill=white, draw=black, shape=rectangle, minimum width=1cm, minimum height=2cm]
\tikzstyle{large morphism2}=[fill=white, draw=black, shape=rectangle, minimum width=2cm, minimum height=1cm]
\tikzstyle{bn}=[fill=black, draw=white, shape=circle, inner sep=1.5pt]
\tikzstyle{wn}=[fill=white, draw=black, shape=circle, inner sep=1.5pt]
\tikzstyle{nn}=[fill=gray, draw=gray, regular polygon, regular polygon sides=3, minimum width=0.3cm, shape border rotate=180, inner sep=0pt]
\tikzstyle{state}=[fill=white, draw=black, regular polygon, regular polygon sides=3, minimum width=0.8cm, shape border rotate=180, inner sep=0pt]
\tikzstyle{effect}=[fill=white, draw=black, regular polygon, regular polygon sides=3, minimum width=0.8cm, shape border rotate=0, inner sep=0pt]
\tikzstyle{medium state}=[fill=white, draw=black, regular polygon, regular polygon sides=3, minimum width=1.3cm, inner sep=0pt, shape border rotate=180]
\tikzstyle{medium effect}=[fill=white, draw=black, regular polygon, regular polygon sides=3, minimum width=1.3cm, inner sep=0pt, shape border rotate=0]
\tikzstyle{large state}=[fill=white, draw=black, regular polygon, regular polygon sides=3, minimum width=2.2cm, shape border rotate=180, inner sep=0pt]
\tikzstyle{treenode}=[fill=white, draw=none, shape=circle]
\tikzstyle{hypergraph wire}=[fill=black, draw=white, shape=circle, inner sep=1.5pt]
\tikzstyle{hypergraph wire red}=[fill=red, draw=white, shape=circle, inner sep=1.5pt]

\tikzstyle{arrow}=[->]
\tikzstyle{dashedarrow}=[->, dashed, draw=gray]
\tikzstyle{trace}=[-, draw={rgb,255: red,132; green,132; blue,132}]
\tikzstyle{tracebox}=[-, fill={rgb,255: red,238; green,238; blue,238}, draw={rgb,255: red,173; green,173; blue,173}]
\tikzstyle{highlight}=[-, draw={rgb,255: red,12; green,42; blue,212}, dashed]

\usepackage{wrapfig}
\usepackage{subcaption}

\pgfdeclarelayer{edgelayer}
\pgfdeclarelayer{nodelayer}
\pgfsetlayers{edgelayer,nodelayer,main}
\tikzstyle{none}=[]

\tikzset{baseline=(current  bounding  box.center)}

\renewcommand{\tikzfig}[1]{
	\InputIfFileExists{#1.tikz}{}{\input{./graphics/#1.tikz}}
}

\tikzset{every picture/.append style={scale=0.3}}
\tikzset{every node/.style={scale=0.75}}

\newtheorem{theorem}{\textbf{Theorem}}
\newtheorem{lemma}[theorem]{\textbf{Lemma}}
\newtheorem{proposition}[theorem]{\textbf{Proposition}}
\newtheorem{corollary}[theorem]{\textbf{Corollary}}
\newtheorem{definition}[theorem]{\textbf{Definition}}
\newtheorem{example}[theorem]{Example}
\newtheorem{remark}[theorem]{Remark}

\usepackage{bm}

\newcommand{\cat}{\mathbf}
\newcommand{\catname}{\mathbf}
\newcommand{\C}{\cat C} 
\newcommand{\R}{\mathbb R}
\newcommand{\N}{\mathbb N}
\newcommand{\keyword}[1]{\mathrm{#1}}
\newcommand{\del}{\keyword{del}}

\newcommand{\id}{\keyword{id}}
\newcommand{\supp}{\keyword{supp}}

\newcommand{\scottbr}[1]{\llbracket{#1}\rrbracket}

\renewcommand{\id}{\keyword{id}}
\newcommand{\tre}[2]{\mathbf{tr}^{#1}_{#2}}
\newcommand{\trace}[1]{\tre{#1}{}}

\newcommand{\contr}[1]{\mathbf{contr}^{#1}}
\newcommand{\hgin}{\keyword{in}}
\newcommand{\hgout}{\keyword{out}}

\newcommand{\finstoch}{\catname{FinStoch}}
\newcommand{\stoch}{\catname{Stoch}}
\newcommand{\borelstoch}{\catname{BorelStoch}}
\newcommand{\gauss}{\catname{Gauss}}
\newcommand{\freemarkov}{\catname{FreeMarkov}}
\newcommand{\sets}{\catname{Sets}}
\newcommand{\hyp}{\catname{Hyp}}
\newcommand{\finhyp}{\catname{FinHyp}}

\newcommand{\mat}{\catname{Mat}}

\newcommand{\eq}{\keyword{e}}
\newcommand{\ff}{\keyword{f}}

\newcommand{\combe}[3]{\langle{#2}|{#3} \rangle_{#1}}
\newcommand{\comb}[2]{\combe{}{#1}{#2}}

\begin{document}
\maketitle

\begin{abstract}
We present a counterexample showing that Markov categories with conditionals (such as BorelStoch) need not validate a natural scheme of axioms which we call contraction identities. These identities hold in every traced monoidal category, so in particular this shows that BorelStoch cannot be embedded in any traced monoidal category. We remedy this under the additional assumption of atomicity: Atomic Markov categories validate all contraction identities, and furthermore admit a notion of trace defined for non-signalling morphisms. We conclude that atomic Markov categories allow for an intrinsic theory of combs without having to assume an embedding into a compact-closed category. 
\end{abstract}

\section{Introduction}\label{sec:introduction}

Markov categories with conditionals have emerged as a general and powerful framework for studying stochastic processes and notions such as conditioning and independence in an abstract way that generalizes reasoning in graphical models (e.g. \cite{Fritz2020b,Fritz2020,Fritz2020a,Fritz2021,Fritz2023c,Jacobs2020a,Stein2021}). An important challenge is to understand the \emph{implicational theory} of such categories: assuming some equation or factorization is valid, which further equations can be derived from this? For example, such quasi-identities lie at the heart of \emph{identifiability} in causal inference; see Figure~\ref{fig:combscausalinfererence} for a treatment of Pearl's front-door adjustment in such categorical terms. 


In this work, we are interested in a particular schema of implications which we call \emph{contraction identities}\footnote{we will shorten quasi-identities to identities from now on}. The simplest instance is the following implication for all morphisms $f_1, f_2 : X \otimes X \to Y$ 
\begin{equation} \scalebox{0.9}{\input{contr-identity.tikz}} \label{cc} \end{equation}

\noindent Every contraction identity is formed by connecting an outgoing wire (here $X$) to an incoming one and pulling the strings tight. This must be done in an acyclic way to make sure the resulting diagram is meaningful -- we need to avoid genuine feedback loops. We will give the full combinatorial definition of the contraction scheme in (\ref{def:contractionid}) using the language of free Markov categories \cite{Fritz2023b} and hypergraphs.

We show using a simple counterexample that \eqref{cc} does not hold in the category $\borelstoch$, despite it having conditionals. We remedy this under a common additional assumption called \emph{atomicity}. Common categories such as $\finstoch,\gauss$ and $\catname{SetMulti}$ are atomic, while $\stoch,\borelstoch$ and categories modelling fresh name generation are not. Our main theorem \ref{propcontractionids} states that atomic Markov categories with conditionals satisfy all contraction identities. As we will discuss next, this enables us to apply the convenient notions of causal traces and combs in such categories. 

\begin{figure}
	\begin{minipage}{0.6\textwidth}
		\[ \scalebox{0.8}{\input{contr-identity-trace.tikz}} \]
		\captionof{figure}{Derivation of \eqref{cc} using traces}
		\label{fig:derive_with_trace}
	\end{minipage}
	\begin{minipage}{0.3\textwidth}
		\[ \input{f-nonsig.tikz} \]
		\captionof{figure}{A morphism $f$ satisfies the non-signalling condition if there exists an $f_s$ such that the equation above holds.}
		\label{fig:non_signalling}
	\end{minipage}
\end{figure}

\paragraph{Causal Traces}

\begin{wrapfigure}{r}{0.4\textwidth}
	\[ \scalebox{0.9}{\input{trace-depiction.tikz}} \]	
	\captionof{figure}{Graphical depiction of the trace operator}
	\label{fig:trace_depiction}
\end{wrapfigure}

A \emph{trace} on a monoidal category $\C$ is an operator which assigns to every morphism $f : X \otimes W \to Y \otimes W$ a morphism $\tre W {X,Y}(f) : X \to Y$ which represents evaluating a feedback loop of $W$ into itself. We depict the trace operation as in Figure \ref{fig:trace_depiction}.

\noindent The trace satisfies a number of axioms which encode the desired properties of such a feedback operation (see Section~\ref{app:traces}), and the graphical calculus reflects these \cite{Selinger2011a}. We show that all contraction identities can be proved from the trace axioms (Proposition~\ref{proptracesoundness}). For example in \eqref{cc}, we can obtain the consequent equation via the trace on Figure \ref{fig:derive_with_trace}. 

As a consequence of our counterexample, $\borelstoch$ cannot be embedded in \emph{any} traced monoidal category. We then establish the following converse result, in trying to obtain a canonical notion of trace on an atomic Markov category. It is hopeless to attempt this for an arbitrary morphism $f : X \otimes W \to Y \otimes W$. It is possible, however, if $f$ satisfies the non-signalling condition, see Figure \ref{fig:non_signalling}. 

This means the input $W$ is not needed to compute the output $W$. So the feedback loop created by the trace remains causal. Our central Definition~\ref{def:ftrace} states that every atomic Markov category with conditionals admits a canonical notion of causal trace, that is a trace defined on non-signalling morphisms. We show that this construction satisfies a restricted version of the trace axioms, and compare with the notion of a partially traced category \cite{Haghverdi2005,Bagnol2015} and traced ideal \cite{Abramsky1999}. The properties of the causal trace suffice to derive all contraction identities (Propositions~\ref{proptracesoundness}~and~\ref{propcontractionids}).

\paragraph{Combs and Causal Inference}

\begin{figure}[h]
	\begin{equation*} \scalebox{0.9}{\input{causal-inference.tikz}} \end{equation*}	
	\caption{Determining the causal effect of $s$ on $c$: Pre-intervention (left) and post-intervention distribution (right). Identifiability means any factorization on the left gives the same post-intervention distribution. We refer to \cite{Jacobs2019} for the details of this example.}\label{fig:combscausalinfererence}
\end{figure}

Combs are a useful tool for studying decompositions of string diagrams into more flexible shapes, with applications in quantum theory and causality \cite{Chiribella2008,kissinger2019,Hefford2022,Roman2020,Roman2020a,Jacobs2019}. For example, Figure~\ref{fig:combscausalinfererence} expresses Pearl's famous \emph{front-door adjustment} \cite[\S3.3.2]{pearl2009causality} using combs and string diagram surgery. 

\noindent Informally, a comb $C$ of type $(A,A') \to (B,B')$ is a diagram of type $A \to A'$ which features a \emph{hole} of type $B \to B'$. This hole can be filled by appropriate morphisms $h$, leading to a composite $C[h]$. This way, a comb describes a second-order process and lends itself to be rendered in an evocative shape (see Figure~\ref{fig:combinsertion}).

\begin{figure}[h]
	\begin{subfigure}{0.3\textwidth}
		\[ \scalebox{0.8}{\input{ex-combs.tikz}} \]
		\caption{Comb insertion (graphically)}
		\label{fig:combinsertion}	
	\end{subfigure}
	\begin{subfigure}{0.25\textwidth}
		\[ \scalebox{0.8}{\input{comb-intensional.tikz}} \]
		\caption{Intensional presentation of a comb as a pair $\comb f g$}
	\end{subfigure}
	\begin{subfigure}{0.45\textwidth}
			\begin{equation} \scalebox{0.8}{\input{combs-contr.tikz}} \label{eq:comb-axiom} \end{equation}	
		\caption{Comb insertion as a contraction identity}
		\label{fig:combinsertion_contrid}
	\end{subfigure}
	\caption{Combs and comb insertion}
\end{figure}

\noindent There are various inequivalent ways of formally defining what a comb is. An extensional definition is to define a comb as a morphism $A \otimes B' \to A' \otimes B$ in $\C$ which is non-signalling from $B'$ to $B$. An intensional definition is as a pair of morphisms $\comb f g$ with $f : A \to E \otimes B$ and $g : E \otimes B' \to A'$, under some kind of equivalence relation. We review these notions in Section~\ref{sec:combs} following \cite{Hefford2022}.

Relating the different definitions of combs to each other is tricky: To go from the extensional definition to comb insertion is an instance of a contraction identity (Figure \ref{fig:combinsertion_contrid}).	

For this reason, the theory of combs has commonly been developed in the setting of compact closed categories. There, second-order processes can be reduced to first-order processes by bending wires, and the various notions of comb are equivalent. The downside is that even when analyzing diagrams in a Markov category $\C$ using combs, we must assume that it comes with an embedding into a compact closed category (for example $\finstoch \hookrightarrow \mat(\R^+)$). Not only is this difficult for practical categories (such an embedding category requires developing a theory of exact conditioning, e.g. \cite{Lavore2023,Stein2021c}), but our counterexample shows that this is generally impossible in the absence of the atomicity axiom: Identity \eqref{eq:comb-axiom} is invalid in $\borelstoch$. We argue that causal traces are sufficient to develop an intrinsic theory of combs. We prove that in every atomic Markov category with conditionals, the extensional and intensional definitions of combs are equivalent.

\paragraph{Contributions}

This work is a shared refinement of three different developments \cite{HoughtonLarsen2021,Hefford2022,Fritz2023a}. The major starting point is \cite{HoughtonLarsen2021}, where the author develops a comprehensive theory of causality and contraction. Their key statement is Lemma 4.2.5, which they prove under the assumption of \emph{universal dilations}. By focusing on Markov categories, we obtain a new entry point into their developments, as we prove Lemma 4.2.5 in our Lemma \ref{lem:atomic-contr} under a different set of assumptions. Using that lemma as a starting point, we rederive a streamlined and self-contained presentation of their constructions which focuses more explicitly on connections with the literature on traced categories. 

We then connect our theory to the different notions of combs compared in \cite{Hefford2022}. We generalize those results significantly from the compact-closed case to a large family of Markov categories, which removes the need to assume an embedding into a compact-closed category as done in \cite{Jacobs2019}. We briefly return to the relationship between universal dilations and combs in Proposition \ref{prop:univ-dil-optic}. Our main contributions are the following.
\begin{enumerate}
\item We recognize the role of the atomicity axiom, which has been introduced in \cite{Fritz2023a}. Our theorems are proven under the assumption of conditionals and atomicity, which are weaker than universal dilations.
\item Our counterexample (\ref{example:borelstochnotatomic}) gives a precise reason for the failure of universality in $\borelstoch$. 
\item The combinatorics of the contraction identities are elegantly phrased using free Markov categories.
\end{enumerate}

\vspace{-0.5cm}

\paragraph{Acknowledgements}

We'd like to thank many people for fruitful discussions, in particular Dylan Braithwaite, Tobias Fritz, Tom{\'a}\v{s} Gonda, Nicholas Gauguin Houghton-Larsen, and Nathaniel Virgo.

\section{Markov Categories}\label{app:markov}

This introductory material is taken from \cite{Fritz2020b}. A \emph{gs-monoidal category} or \emph{CD-category} (due to \cite{Cho2019}) is a symmetric monoidal category $(\C,\otimes,I)$ where every object is coherently equipped with the structure of a commutative comonoid $\Delta_X : X \to X \otimes X$, $\del_X : X \to I$. We render these graphically as on Figure \ref{fig:markov_structure}.

\begin{figure}[ht]
	\begin{minipage}{0.45\textwidth}
		\[ {\input{markov-structure.tikz}} \]
		\caption{Graphical depiction of copy and delete maps}
		\label{fig:markov_structure}
	\end{minipage}
	\hspace{0.05\textwidth}
	\begin{minipage}{0.45\textwidth}
		\[ {\input{def-det.tikz}} \]
		\caption{A map $f$ is \emph{copyable} (resp. \emph{discardable}) if the equation on the left (resp. right holds)}
	\end{minipage}
\end{figure}

 \noindent A \emph{Markov category} is a CD category in which every morphism is discardable. Equivalently, this means $\del$ is a natural transformation, and therefore $I$ is a final object. Copying is not assumed to be natural. From now on, let $\C$ be a Markov category. 
%
%

We say that $\C$ \emph{has conditionals} if for every $f : A \to X \otimes Y$ there exists $f|_X : X \otimes A \to Y$ such that the left equation of Figure \ref{fig:cond_bayesianinv} holds.

For morphisms $p : A \to X$ and $f : X \to Y$, a \emph{Bayesian inverse} is a morphism $f^\dagger_p : A \otimes Y \to X$ such that the right equation on Figure \ref{fig:cond_bayesianinv} holds.

\begin{wrapfigure}{r}{0.6\textwidth}
	\begin{minipage}{0.27\textwidth}
		\[ \scalebox{0.8}{\input{def-cond.tikz}} \]		
	\end{minipage}
	\begin{minipage}{0.27\textwidth}
		\[ \scalebox{0.8}{\input{def-dagger.tikz}} \]	
	\end{minipage}
	\caption{Conditionals and Bayesian inverses}
	\label{fig:cond_bayesianinv}
\end{wrapfigure}

Conditionals and Bayesian inverses are mutually interdefinable, and we will be using both in the present article. 

\subsection{Example Categories}

We briefly recall the Markov categories of interest, mainly to establish notation. For the full definitions, we ask the reader to consult the references.

\begin{example}
	The Markov category $\finstoch$ consists of
	\begin{enumerate}
		\item objects are finite sets $X$
		\item morphisms $p : X \to Y$ are stochastic matrices, with entries written $p(y|x) \in [0,1]$, subject to the axiom
		\[ \forall x \in X, \sum_{y \in Y} p(y|x) = 1 \]
		\item composition is matrix multiplication, a.k.a. the Kolmogorov-Chapman equation
		\[ (gf)(z|x) = \sum_y g(z|y)f(y|x)\]
	\end{enumerate}
\end{example}
\noindent The compact-closed category $\mat(\R^+)$ is defined like $\finstoch$ but allows arbitrary matrices with nonnegative entries as morphisms. $\finstoch$ corresponds to the subcategory of $\mat(\R^+)$ of morphisms which are discardable. 

\begin{example}
	The Markov category $\borelstoch$ consists of
	\begin{enumerate}
		\item morphisms are standard Borel spaces $(X,\Sigma_X)$
		\item morphisms $p : X \to Y$ are Markov kernels $p : X \times \Sigma_Y \to [0,1]$. We write $p(x)$ for the measure $p(x,-) : \Sigma_Y \to [0,1]$ on $Y$. 
		\item composition is Lebesgue integration
		\[ (gf)(x,E) = \int g(y,E) f(x,\mathrm dy)\]
	\end{enumerate}
\end{example}
\noindent For $x \in X$, we write $\delta_x$ for the Dirac distribution centered at $x$. We can consider every measurable function $f: (X,\Sigma_X) \to (Y,\Sigma_Y)$ as a $\borelstoch$ morphism $\delta_f$ defined by $\delta_f(x) = \delta_{f(x)}$. As a slight abuse of notation, we will often write $f$ for both the function and its induced Markov kernel. \\

\noindent Markov categories capture a variety of flavors of probability and nondeterminism. The following categories will only be mentioned in examples, so we give the references: $\catname{SetMulti}$ (sets and left-total relations) \cite{Fritz2020b}, $\catname{TychStoch}$ (Tychonoff spaces and continuous Markov kernels) \cite{Fritz2023a}, $\catname{Gauss}$ (Markov kernels built from multivariate normal distributions and linear maps) \cite{Fritz2020b}.

\section{Traced Monoidal Categories}\label{app:traces} 

We briefly recall the notion of a traced monoidal category \cite{Joyal1996}.

\noindent
A traced monoidal category is a symmetric monoidal category $\C$ together with a family of operators
\[ \tre W {X,Y} : \C(X \otimes W,Y \otimes W) \to \C(X,Y) \]
satisfying the axioms depicted in Figure \ref{fig:trace_axioms}. We will omit the grey shading of the trace boxes when it is clear from context how to interpret the trace.
\begin{figure}[ht]
	\begin{subfigure}{0.45\textwidth}
		\[ \scalebox{0.8}{\input{tightening-axiom.tikz}} \]
		\caption{Tightening (naturality in $X,Y$): For all $f : X \otimes W \to Y \otimes W$, $g : X' \to X$, $h : Y \to Y'$, we have 
			$\trace W((h \otimes \id_W) \circ f \circ (g \otimes W)) = h \circ \trace W(f) \circ g$ }
	\end{subfigure}
	\hspace{0.1\textwidth}
	\begin{subfigure}{0.45\textwidth}
		\[ \scalebox{0.8}{\input{sliding-axiom.tikz}} \]
		\caption{Sliding (dinaturality in $W$): For all $f : X \otimes W \to Y \otimes W'$ and $g : W' \to W$, we have
			$\trace {W}((\id_Y \otimes g) \circ f) = \trace {W'}(f \circ (\id_X \otimes g))$ }
	\end{subfigure}
	\begin{subfigure}{0.45\textwidth}
		\[ \scalebox{0.8}{\input{vanishing-axiom.tikz}} \]
		\caption{Vanishing (coherence with $I$). For all $f : X \otimes I \to Y \otimes I$, we have
			$\trace I(f) = (X \xrightarrow{\rho_X^{-1}} X \otimes I \xrightarrow f Y \otimes I \xrightarrow{\rho_Y} Y)$}
	\end{subfigure}
	\hspace{0.1\textwidth}
	\begin{subfigure}{0.45\textwidth}
		\[ \scalebox{0.8}{\input{associativity-axiom.tikz}} \]
		\caption{Associativity (coherence with $\otimes$): For all $f : X \otimes U \otimes V \to Y \otimes U \otimes V$ we have 
			$\trace {U \otimes V}(f) = \trace U(\trace V(f))$}
	\end{subfigure}
	\begin{subfigure}{0.45\textwidth}
		\[ \scalebox{0.8}{\input{superposition-axiom.tikz}} \]
		\caption{Superposition (strength): For all $f : X \otimes W \to Y \otimes W$ and $g : X' \to Y'$, we have 
			$\trace W(g \otimes f) = g \otimes \trace W(f)$}
	\end{subfigure}
	\hspace{0.1\textwidth}
	\begin{subfigure}{0.45\textwidth}
		\[ \scalebox{0.8}{\input{yanking-axiom.tikz}} \]
		\caption{Yanking: $\trace W(\keyword{swap}_{W,W}) = \id_W$}
	\end{subfigure}
	\caption{Axioms of trace}
	\label{fig:trace_axioms}
\end{figure}

\section{Atomic Markov Categories}\label{sec:atomicmc}

We recall the notion of atomicity in Markov categories, as well as the prerequisite notions of almost sure equality and absolute continuity found in \cite{Fritz2023a}. Proofs of our propositions are to be found in the appendix. Let $\C$ be a Markov category. 

\begin{wrapfigure}{r}{0.3\textwidth}
	\[ \scalebox{0.8}{\input{def-ase.tikz}} \]
	\caption{Almost sure equality}
	\label{fig:almost_sure_equality}
\end{wrapfigure}

\begin{definition}[Almost sure equality -- {\cite{Fritz2023a}}]
Given a morphism $p : A \to X$, we say $f_1,f_2 : W \otimes X \to Y$ are \emph{$p$-almost surely equal} (written $f_1 =_p f_2$) if the equation in Figure \ref{fig:almost_sure_equality} holds.
\end{definition}

\begin{definition}[Absolute continuity -- {\cite{Fritz2023a}}]
For two morphisms $p : A \to X$ and $q : B \to X$, we say that $p$ is \emph{absolutely continuous with respect to $q$} (written $p \ll q$), if $f_1 =_q f_2$ implies $f_1 =_p f_2$ for all $f_1,f_2 : W \otimes X \to Y$.
\end{definition}
\noindent We note the quantification over arbitrary $W$ which is stronger than earlier definitions (e.g. \cite{Fritz2020a}). These abstract notions capture the usual definitions of almost sure equality in our example categories (see \cite[Section~2.2]{Fritz2023a}):
\begin{enumerate}
\item in $\finstoch$, $f_1,f_2$ are $p$-almost surely equal if $f_1(y|x) = f_2(y|x)$ for all $x$ with $p(x|a)>0$ for some $a \in A$. We have $p \ll q$ if $\supp(p) \subseteq \supp(q)$ where
\[ \supp(p) = \{ x \in X : \exists a \in A, p(x|a) > 0 \} \]
\item in $\borelstoch$, $f_1,f_2$ are $p$-almost surely equal if the set $D=\{ x : f_1(x) \neq f_2(x) \}$ satisfies $p(D|a) = 0$ for all $a \in A$. We have $p \ll q$ if for every every measurable subset $S \subseteq X$, we have
\[ (\forall b \in B, q(S|b) = 0) \Rightarrow \forall a \in A, p(S|a) = 0 \]
\end{enumerate}

\begin{definition}[Atomicity -- \cite{Fritz2023a}]\label{def:atomic}
We call a morphism $p : A \to X$ \emph{atomic} if $\Delta_X \circ p \ll p \otimes p$. We call the Markov category $\C$ atomic if every morphism in it is atomic. 
\end{definition}

\noindent The following characterizations are known from \cite{Fritz2023a}, but we repeat them because they are instructive.

\begin{example}
For every morphism $p : A \to X$ in $\finstoch$, we have 
\[ (p \otimes p)(x_1,x_2|a_1,a_2) = p(x_1|a_1)p(x_2|a_2), \qquad (\Delta_X \circ p)(x_1,x_2|a) = \begin{cases} 
p(x_1|a), &x_1 = x_2 \\
0, &\text{otherwise}
\end{cases} \]
so it holds that $\supp(\Delta_X \circ p) \subseteq \supp(p \otimes p)$. Therefore $\finstoch$ is atomic.
\end{example}

\noindent The name \emph{atomicity} suggests that this property fails for distributions which are atomless, such as the Lebesgue measure. For a morphism $p : A \to X$ in $\borelstoch$, we define its set of atoms $\mathcal A \subseteq X$ as 
\[ \mathcal A = \{ x \in X : \exists a \in A, p(\{x\}|a) > 0 \}\]
We call $p$ \emph{completely atomic} if $p(\mathcal A|a) = 1$ for all $a \in A$, i.e. its probability mass is fully concentrated on its atoms. It is shown in \cite[Theorem~3.2.7]{Fritz2023a} that $p$ is atomic in the sense of Definition~\ref{def:atomic} if and only if it is completely atomic. It is easy to see that the Lebesgue measure $\nu$ on the interval $[0,1]$ is \emph{not} atomic, as its set of atoms $\mathcal A$ is empty. We will elaborate a concrete counterexample, showing that $\Delta \circ \nu \not \ll \nu \otimes \nu$.

\begin{example}\label{example:borelstochnotatomic}
$\borelstoch$ is not atomic. 
\end{example}
\begin{proof}
Let $X=[0,1]$ and $\nu : I \to X$ be the Lebesgue measure. We will construct two measurable functions which are almost-sure equal with respect to $\nu \otimes \nu$ (the Lebesgue measure on the square), but not with respect to $\Delta \circ \nu$ (the pushforward measure on the diagonal). 

We take these functions to be the characteristic functions of the empty set and the diagonal respectively, i.e. $\eq, \ff : X \otimes X \to \{0,1\}$ given by $\eq(x,y)=[x=y]$, $\ff(x,y)=0$.
Then $\eq =_{\nu \otimes \nu} \ff$, because the set $\{ (x,y) : \eq(x,y) \neq \ff(x,y)\}$ has measure $0$ under $\nu \otimes \nu$. The following equality holds in string diagrams
\[ \scalebox{0.8}{\input{borelstoch-atomic-1.tikz}}\] 
However $\eq \neq_{\Delta \circ \nu} \ff$, because the set $\{ (x,y) : \eq(x,y) \neq \ff(x,y)\}$ has measure $1$ under $\Delta \circ \nu$! Similarly, the following string diagrams are \emph{not} equal.
\[ \scalebox{0.8}{\input{borelstoch-atomic-2.tikz}}\]
This shows $\Delta \circ \nu \not \ll \nu \otimes \nu$. 
\end{proof}

\noindent As a simple corollary, we obtain the following:

\begin{example}\label{example:borelstoch-no-contraction}
$\borelstoch$ does not satisfy the contraction identity \eqref{cc} from the introduction.
\end{example}
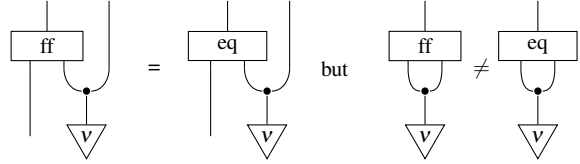
\begin{proof}
\[ \scalebox{0.8}{\input{contr-identity-borelstoch.tikz}} \]
\end{proof}

\noindent This type of counterexample can easily be adapted to other Markov categories which feature atomless distributions, such as quasi-Borel spaces or nominal sets (\cite{Stein2021}). As we will see, such counterexamples are impossible when the category in question is atomic.

A subtle caveat is that the notion of atomicity is dependent on the surrounding Markov category: The standard normal distribution $\mathcal{N}(0,1)$ is non-atomic in $\borelstoch$, but it \emph{is} atomic in the subcategory $\catname{Gauss}$. More generally, the category $\catname{TychStoch}$ of Tychonoff spaces and continuous Markov kernels is atomic \cite{Fritz2023a}, despite featuring atomless measures: Note that the map $\eq$ from the counterexample \ref{example:borelstoch-no-contraction} fails to be continuous. \\

\noindent We return to the abstract properties of atomic morphisms.
\begin{proposition}\label{prop:atomicmorphisms}
In a Markov category $\C$, the following morphisms $p$ are always atomic
\begin{enumerate}
\item morphisms of `full support', in the sense that $f_1 =_p f_2$ implies $f_1 = f_2$.
\item deterministic morphisms
\item if $\C$ satisfies the causality axiom (\cite[11.31]{Fritz2020b}), $g$ is deterministic, and $p$ is atomic, then the composite $g \circ p$ is atomic
\end{enumerate}
\end{proposition}
\noindent Note that atomic morphisms are generally not closed under composition, and atomicity cannot be checked on points (see Example~\ref{ex:atomic-no-composition}). We now show that the atomicity axiom plus conditionals imply a basic contraction identity. These will be sufficient to prove all of them (\ref{sec:causaltraces-laws}). We need the following technical notion to rule out pathological cases:

\begin{definition}
An object $W$ of a Markov category is called \emph{cancellable} if $\del_W \otimes f_1 = \del_W \otimes f_2$ implies $f_1 = f_2$. We call the Markov category \emph{cancellative} if every object $W$ is cancellable.
\end{definition}
\noindent This condition is called \emph{normality} in \cite{HoughtonLarsen2021}. In practical examples, most objects are cancellable, for example any object that admits a state $I \to W$. In $\finstoch$ and $\borelstoch$, every object except $W=\emptyset$ is cancellable. We may formally restrict our attention to sub-Markov categories $\finstoch^*, \borelstoch^*$ on non-empty objects, or simply assume cancellability when needed. 

\begin{lemma}\label{lem:disint-indep}
In an atomic Markov category, we have for all $f_1,f_2 : X \to W$ and $g_1,g_2 : W  \otimes X \otimes W \to Y$ that
\[ \scalebox{0.8}{\input{lem-statement.tikz}} \]
\end{lemma}

\begin{lemma}\label{lem:atomic-contr}
Let $\C$ be an atomic Markov category with conditionals. Then
\begin{equation*} \scalebox{0.8}{\input{combs-contr.tikz}} \end{equation*}
\end{lemma}

\noindent Note that this is precisely the statement of Lemma 4.2.5 of \cite{HoughtonLarsen2021} in our setting.

\section{Causal Traces}\label{sec:causaltraces}

\begin{definition}
We say a morphism $f : X \otimes W' \to Y \otimes W$ is \emph{non-signalling} (from $W'$ to $W$) if there exists a morphism $f_s : X \to W$ such that 
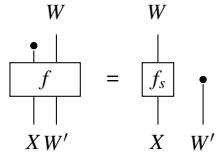
\[ \scalebox{0.8}{\input{f-nonsig-2.tikz}} \]
This captures the intuition that in order to determine $W$, we don't need access to $W'$ (but we do to compute the joint output $Y$).
\end{definition}

\noindent For non-signalling morphisms, we can hope to define a canonical trace as follows.

\begin{proposition}
Let $\C$ be a Markov category with conditionals. Then $f : X \otimes W' \to Y \otimes W$ is non-signalling if and only if it can be written in the following form
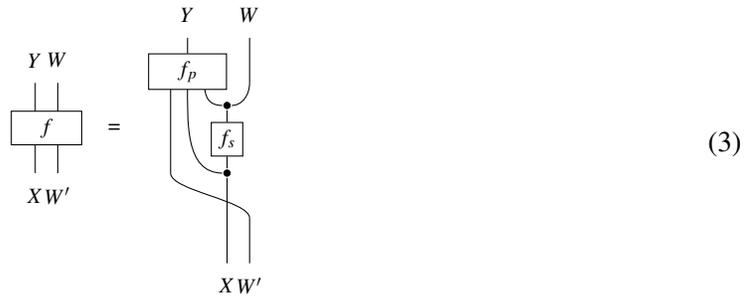
\begin{equation} \scalebox{0.8}{\input{f-nonsig-nf.tikz}} \label{def:disint} \end{equation}
In this article, we will call the form \eqref{def:disint} a \emph{disintegration} of $f$. Note that while $f_s$ is unique (assuming cancellability), $f_p$ is not.
\end{proposition}

\begin{definition}
Let $\C$ be an atomic Markov category with conditionals, and $f : X \otimes W \to Y \otimes W$ be non-signalling. Then we define the causal trace $\trace W(f) : X \to Y$ in terms of any disintegration as
\begin{equation} \scalebox{0.8}{\input{ftrace.tikz}} \label{def:ftrace} \end{equation}
By Lemma~\ref{lem:disint-indep}, this definition does not depend on the choice of disintegration. 
\end{definition}

\noindent The definition of the causal trace is unique in that it is forced via the trace axiom. For a non-signalling morphism $f$, \emph{any} trace must send $f$ to \eqref{def:ftrace}. The special role of non-signalling is also reflected in the following proposition.

\begin{proposition}\label{prop:trace-discardable}
Let $\C$ be a traced semicartesian monoidal category. If $f : X \otimes W \to Y \otimes W$ is a discardable morphism, then $\trace W (f) : X \to Y$ need not be discardable, unless $f$ is non-signalling.
\end{proposition}

\noindent We spell out concretely how to compute the trace in $\finstoch$:

\begin{example}
Let $f : X \otimes W \to Y \otimes W$ be a non-signalling morphism in $\finstoch$, and assume $W \neq \emptyset$. As we can compute the causal trace using any embedding into a traced category, such as the usual embedding into $\mat(\R)$. There we have
\[ \trace W (f)(y|x) = \sum_{w} f(y,w|x,w) \]
This will not define a normalized probability kernel for a general $f$. However, if $f$ is non-signalling then the sum $\sum_{y} f(y,w|x,w')$ does not depend on $w'$. Hence by fixing any $w_0 \in W$, we can prove normalization
\begin{align*}
\sum_{y} \trace W f(y|x) = \sum_y \sum_w f(y,w|x,w) 
= \sum_w \sum_y f(y,w|x,w) 
= \sum_w \sum_y f(y,w|x,w_0) 
= 1
\end{align*}
from the normalization of $f$. 
\end{example}

\noindent In the language of \cite{Virgo2023}, we can see a non-signalling morphism as a degenerate kind of automaton whose future states do not depend on the current states. The causal trace takes a fixed point by reusing identical copies of states. 

\subsection{Laws of the Causal Trace}\label{sec:causaltraces-laws}

We now state and prove a list of properties which the causal trace enjoys. We will compare it with the notion of \emph{partial trace} discussed by \cite{Haghverdi2005,Bagnol2015} and \emph{traced ideal} introduced in \cite{Abramsky1999}. This requires us to postulate a class of traceable morphisms, which we choose as the non-signalling ones, $\mathbb T_W^{X,Y} = \{ f : X \otimes W \to Y \otimes W \mid \text{$f$ is non-signalling}\}$. 
All proofs are given in the Appendix (Section~\ref{app:traceproofs}).

\begin{proposition}[Tightening -- naturality in $X,Y$]\label{prop:tightening}
If $f : X \otimes W \to Y \otimes W$ is traceable, then $(h \otimes \id_W) \circ f \circ (g \otimes \id_W) : X' \otimes W \to Y' \otimes W$ is traceable and 
	\begin{equation}
		\trace W((h \otimes \id_W)\circ f \circ (g \otimes \id_W)) = h \circ \trace W(f) \circ g
	\end{equation}
\end{proposition}

\begin{proposition}[Sliding -- dinaturality in $W$]\label{prop:sliding}
Let $f : X \otimes W' \to Y \otimes W$ and $g : W \to W'$. If $f$ is traceable, then so are $(\id_Y \otimes g) \circ f$ and $f \circ (\id_X \otimes g)$, and
\begin{equation}
\trace U((\id_Y \otimes g) \circ f) = \trace V(f \circ (\id_X \otimes g))
\end{equation}
\end{proposition}
\begin{remark}
	The corresponding partial trace axiom stipulates that $(\id_Y \otimes g) \circ f$ is traceable if and only if $f \circ (\id_X \otimes g)$ is traceable, but this does not require that traceability of $f$ implies traceability of the composites. Our property is the same as the corresponding traced ideal axiom. 
\end{remark}

\begin{proposition}[Vanishing (Coherence with $I$)]
\label{prop:vanishing}
All $g : X \otimes I \to Y \otimes I$ are traceable, and $\trace I(g) = g$.
\end{proposition}

\begin{proposition}[Associativity (Coherence with $\otimes$)] \label{prop:associativity}
If $f : X \otimes (U \otimes V) \to Y \otimes (U \otimes V)$ is traceable, then so are $f : (X \otimes U) \otimes V \to (Y \otimes U) \otimes V$ and $\trace V(f) : X \otimes U \to Y \otimes U$, and
	\begin{equation*}
		\trace {U \otimes V}(f) = \trace U(\trace V(f))
	\end{equation*}
\end{proposition}
\begin{remark}
	This property is stronger than the corresponding partial trace axiom, but weaker than the traced ideal axiom.	
\end{remark}

\begin{proposition}[Superposition (Strength)]\label{prop:superposition}
	If $f : X \otimes W \to Y \otimes W$ is traceable, so is $g \otimes f : X' \otimes X \otimes W \to Y' \otimes Y \otimes W$, and
	\begin{equation}
		\trace W(g \otimes f) = g \otimes \trace W(f)
	\end{equation}
\end{proposition}

\begin{proposition}[Yanking]\label{prop:yanking}
The symmetry is traceable, and
	\begin{equation}
		\trace W(\mathrm{swap}_{W,W}) = \id_W
	\end{equation}
\end{proposition}
\begin{remark}
	Our property is the same as the corresponding partial trace axiom. This is stronger than the yanking axiom for traced ideals, as there the symmetry is not required to be traceable.
\end{remark}

\section{Causal Traces in Free Markov Categories}\label{secfreemarkovtrace}

In this section, we show that free Markov categories also have causal traces, and that interpretations of Markov string diagrams in cancellative, atomic Markov categories $\C$ with conditionals preserve causal traces. This is an analogue of Theorem 4.2.13 of \cite{HoughtonLarsen2021}, albeit in a different setting: Markov categories with extra structure instead of universal theories. In particular, we obtain that all contraction identities are satisfied. That is, in the language of \cite{HoughtonLarsen2021}, the causal trace of a non-signalling morphism in $\C$ can be computed via any stencil representation. 

We use free Markov categories as constructed in \cite{Fritz2023b}. Recall that every monoidal signature $\Sigma$ gives rise to a finite hypergraph which we also denote by $\Sigma$. A labelling of wires and boxes in a hypergraph $G$ is a hypergraph homomorphism to $\Sigma$. Finite hypergraphs and their homomorphisms form a category $\finhyp$. A \emph{Markov string diagram} is a cospan $\underline{m} \xrightarrow{i} G \xleftarrow{j} \underline{n}$ with discrete hypergraphs $\underline{m}, \underline{n}$, satisfying the conditions of acyclicity, left monogamy, and having no eliminable boxes. The setting is recalled in more detail in Appendix~\ref{app:freemarkov}.

\begin{proposition}[\cite{Fritz2023b}]
	The free Markov category $\freemarkov_\Sigma$ over a monoidal signature $\Sigma$ can be constructed as follows:
	\begin{enumerate}
		\item Objects are hypergraph homomorphisms $\underline{m} \to \Sigma$, i.e. lists of types in $\Sigma$.
		\item Morphisms are (isomorphism classes of) Markov string diagrams, which compose by pushout and subsequent normalisation (elimination of eliminable boxes).
		\item The tensor is given by coproduct.
	\end{enumerate}
\end{proposition}
\noindent To ease notation we will write $f : n \to m$ for a morphism $f : (\underline{m} \to \Sigma) \to (\underline{n} \to \Sigma)$, thus leaving the labelling implicit. It is straightforward to verify the following:

\begin{proposition}
A morphism $f : n \otimes w \to m \otimes w$ in $\freemarkov_\Sigma$ is non-signalling if and only if there are no directed paths from input ports in $w$ to output ports in $w$.
\end{proposition}

\noindent We now show that $\freemarkov_\Sigma$ has causal traces, by 
describing the appropriate combinatorial contraction as a contraction 
of hypergraphs: Let $f : m \otimes w \to n \otimes w$ be 
non-signalling, represented by the cospan $\underline{m} + 
\underline{w} \xrightarrow{[i,i']} G \xleftarrow{[j,j']} 
\underline{n}+\underline{w}$. Graphically, forming the contraction 
amounts to gluing the wires connected to matching input and output 
ports in $w$, thus making them inner wires (not connected to any 
input or output port), and normalizing. See Figure 
\ref{figcontrexample} for an illustration. The need for normalisation 
is apparent from the example on Figure \ref{fig:need-to-normalise}.

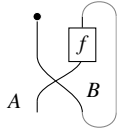
\begin{figure}[ht]
	\begin{minipage}{0.4\textwidth}
		\[ \input{contr-normalisation.tikz}\]
		
		\caption{An example where normalisation is needed to compute the trace}
		\label{fig:need-to-normalise}
	\end{minipage}
	\begin{minipage}{0.55\textwidth}
		\[\begin{tikzcd}
			&& {\contr{w}(G)} & G & {\underline{n}} \\
			& G && {\underline{w}} \\
			{\underline{m}} && {\underline{w}+\underline{w}}
			\arrow["{[\id,\id]}", from=3-3, to=2-4]
			\arrow["{[i',j']}"', from=3-3, to=2-2]
			\arrow[dashed, from=2-2, to=1-3]
			\arrow[dashed, from=2-4, to=1-3]
			\arrow["\lrcorner"{anchor=center, pos=0.125, rotate=-45}, draw=none, from=1-3, to=3-3]
			\arrow["{i}", from=3-1, to=2-2]
			\arrow["{j}"', from=1-5, to=1-4]
			\arrow[dashed, from=1-4, to=1-3]
		\end{tikzcd}\]
		\caption{Construction of syntactic traces}
		\label{fig:contr_as_pushout}
	\end{minipage}
\end{figure}

\noindent More formally, we define the contracted string diagram $\contr{w}(f) : m \to n$ as the normalization of the resulting cospan in the diagram of Figure \ref{fig:contr_as_pushout}, where the central square is a pushout.

\noindent We verify that the resulting string diagram is acyclic and left monogamous. Acyclicity follows from the non-signalling assumption. For left monogamy, observe that after gluing the only affected wires are the one connected to ports in $w$. Every wire connected to an output port is either also connected to an input port, or is an output of a box. In both cases, left monogamy is preserved, for if a wire was both connected to an input and an output port, the input port cannot be in $w$ by the acyclicity assumption.

\begin{figure}[h]
	\begin{minipage}{0.4\textwidth}
		\[ \scalebox{0.8}{\input{context-trace.tikz}} \]
	\end{minipage}
	\begin{minipage}{0.25\textwidth}
		\[ \scalebox{0.8}{\input{hypergraph-example-1.tikz}} \]
	\end{minipage}
	$\rightsquigarrow$
	\begin{minipage}{0.25\textwidth}
		\[ \scalebox{0.8}{\input{hypergraph-example-2.tikz}} \]
	\end{minipage} 
	\caption{An example of computing the contraction in the hypergraph representation. (Left) The causal trace we want to compute (Right) Hypergraph representations. Black dots represent wires, white dots ports; the red ports are being contracted.} \label{figcontrexample}
\end{figure}

\begin{proposition}
The contraction operation $\contr{}$ satisfies the causal trace axioms (as stated in Propositions \ref{prop:tightening} - \ref{prop:yanking}) for $\freemarkov_\Sigma$.
\end{proposition}
\begin{proof}
	Vanishing, strength, and yanking are immediate. The other axioms require bit more caution around the normalization step that occurs in sequential composition and contraction. We can use the fact that normalization of string diagrams is an identity-on-objects gs-monoidal functor $\cat{FreeGS}_\Sigma \to \freemarkov_\Sigma$ from the free gs-monoidal category $\cat{FreeGS}$ (\cite{Fritz2023b}, Lemma 6.5) to solve these cases.
\end{proof}

\noindent We finally are ready to formally define our notion of contraction identities.

\begin{definition}\label{def:contractionid}
	A Markov category $\C$ is said to satisfy all contraction identities if for all $C_1,C_2 : m \otimes w \to n \otimes w$ non-signalling in a free Markov category over any signature $\Sigma$, and all interpreting functors $\scottbr{-} : \freemarkov_\Sigma \to \C$, we have that if $\scottbr{C_1} = \scottbr{C_2}$, then $\scottbr{\contr{w}(C_1)} = \scottbr{\contr{w}(C_2)}$.
\end{definition}

This definition is analogous to the `notions of contraction' in \cite{HoughtonLarsen2021}. We can show by induction that if $\C$ already has causal traces, then any Markov functor from $\freemarkov_\Sigma$ must preserve them: 

\begin{proposition}\label{proptracesoundness}
Let $\C$ have causal traces, and let $f : m \otimes w \to n \otimes 
w$ in $\freemarkov_\Sigma$, \\ $\scottbr{-} : \freemarkov_\Sigma \to 
\C$ an interpreting functor. Then, if $f$ is non-signalling, so is 
$\scottbr{f}$, and \\ $\scottbr{\contr{w}(f)} = 
\trace{\scottbr{w}}(\scottbr{f})$.
\end{proposition}
\begin{proof}
	The first point is clear. For the second point we proceed by strong induction on $w$. If $w = 0$, then $f$ is of the form $g \otimes \id_I$. We are done by vanishing (\ref{prop:vanishing}).
	
	Now assume the statement for all $c < k+1$. Let $f : n + k + 1 \to m + k + 1$. 
	\[ \trace{\scottbr{k+1}}(\scottbr{f}) = \trace{\scottbr{k}}(\trace{\scottbr{1}}(\scottbr{f})) =
	\trace{\scottbr{k}}(\scottbr{\contr{1}(f)} =
	\scottbr{\contr{k}(\contr{1}(f))} = \scottbr{\contr{k+1}(f)} \]
	We applied the induction hypothesis twice, and associativity (\ref{prop:associativity}) twice.
\end{proof}

As a direct consequence of the previous proposition we obtain:
\begin{corollary}\label{propcontractionids}
Every cancellative atomic Markov category with conditionals satisfies all contraction identities.
\end{corollary}

\section{Notions of Combs}\label{sec:combs}

We recall various definitions of combs and refer to \cite{Hefford2022,Roman2020,Roman2020a} for reference.

\begin{definition}[Comb]
Let $\C$ be a symmetric monoidal category. A comb of type $(A,A') \to (B,B')$ is a triple $C = \combe E f g$ consisting of an object $E$ and morphisms $f : A \to E \otimes B$ and $g : E \otimes B' \to A'$. We will omit the subscript $E$ if it is clear from context. For a morphism $h : B \otimes K \to B' \otimes K'$, the comb insertion $C[h] : A \otimes K \to A' \otimes K'$ is defined as $C[h] = (f \otimes \id_K) ; (\id_E \otimes h) ; (g \otimes \id_{K'})$. The \emph{extension} of the comb $C$ is the joint morphism $C[\keyword{swap}_{B,B'}] : A \otimes B' \to A' \otimes B$. 
\end{definition}

\begin{definition}\label{defcombequivalence}
Two combs $C_1,C_2 : (A,A') \to (B,B')$ are called 
\begin{enumerate}
\item \emph{extensionally equivalent} if their extensions are equal
\item \emph{contextually equivalent} if for all $h : B \otimes K \to B' \otimes K'$, we have $C_1[h] = C_2[h]$
\item \emph{optically equivalent} if they are identified as elements of the coend
\[ \int^E \C(A,E \otimes B) \times \C(E \otimes B', A') \]
Concretely, this is the equivalence generated by `sliding' for all $s : E \to E'$
\[ \combe{E'}{f ; (s \otimes \id_B)}{g} \sim \combe{E}{f}{(s \otimes \id_{B'}) ; g}\]
\end{enumerate}
\end{definition}

\begin{proposition}[\cite{Hefford2022}]
Optic equivalence implies contextual equivalence. Contextual equivalence implies extensional equivalence. 
\end{proposition}

\noindent The converses are false in general. Comb insertion need not be well-defined under extensional equivalence; counterexample atomicity. It is shown in \cite{Hefford2022} that all three notions coincide in two special cases: compact closed categories and cartesian categories. We generalize this as follows:

\begin{proposition}\label{prop:extensional-implies-contextual}
In any category with causal traces, extensional and contextual equivalence coincide. 
\end{proposition}
\begin{proof}
We can compute the comb insertion from the extension using the causal trace from Figure \ref{figcontrexample}.
\end{proof}

\begin{theorem}\label{thm:combs-nonsignalling}
Let $\C$ be an atomic Markov category with conditionals. Then extensional and contextual equivalence coincide. Computing the extension defines a bijection between
\begin{enumerate}
\item morphisms $f : A \otimes B' \to A' \otimes B$ that are non-signalling 
\item contextual equivalence classes of combs $(A,A') \to (B,B')$
\end{enumerate}
\end{theorem}

\noindent This theorem generalizes the examples covered by \cite{Hefford2022} to include many common Markov categories without assuming an embedding in a compact-closed category. The case of optic equivalence requires further structure.

\begin{proposition}\label{prop:univ-dil-optic}
If $\C$ has universal dilations in the sense of \cite{HoughtonLarsen2021}, then extensional and optic equivalence coincide. 
\end{proposition}

\noindent Because $\finstoch$ has universal dilations (\cite[Theorem~2.4.6]{HoughtonLarsen2021}), this means that for $\finstoch$, all notions of combs coincide. 

\section{Conclusions and Future Work}\label{sec:related}

We have shown that the contraction identities hold in every cancellative atomic Markov category with conditionals, and used this to develop a theory of causal traces and relate various notions of comb equivalence. Our work leaves an array of interesting open questions, including about converses of our results:

When does a Markov category embed into a traced category? It is known that every partially traced category embeds in a traced one \cite{Bagnol2015}, but our axioms for the causal trace do not imply all partial trace axioms. We conjecture that the difference is similar to what Houghton-Larsen achieves with his construction of \emph{causal channels} in \cite[Section~4.1]{HoughtonLarsen2021}.

To what extent is atomicity a necessary assumption? Does an embedding in a traced category imply atomicity? Atomicity is intimately related to supports \cite{Fritz2023a}, and in turn to universal dilations, though the precise relationship remains to be clarified.

Our developments in Section~\ref{secfreemarkovtrace} suggest that the information-flow properties of free Markov categories are an interesting area of study beyond this work: We conjecture that these categories are atomic, and have supports but no conditionals. We also believe that free Markov categories embed in free hypergraph categories, even though this point was sidestepped in \cite{Fritz2023b}.

We would also like to explore if failure of atomicity poses formal challenges in causal inference, given that the approach to causal inference in \cite{Jacobs2019} relied on combs, or if the approach can be refined to not rely on them.

\bibliographystyle{eptcs}
\bibliography{main}

\begin{thebibliography}{10}
\providecommand{\bibitemdeclare}[2]{}
\providecommand{\surnamestart}{}
\providecommand{\surnameend}{}
\providecommand{\urlprefix}{Available at }
\providecommand{\url}[1]{\texttt{#1}}
\providecommand{\href}[2]{\texttt{#2}}
\providecommand{\urlalt}[2]{\href{#1}{#2}}
\providecommand{\doi}[1]{doi:\urlalt{https://doi.org/#1}{#1}}
\providecommand{\eprint}[1]{arXiv:\urlalt{https://arxiv.org/abs/#1}{#1}}
\providecommand{\bibinfo}[2]{#2}

\bibitemdeclare{article}{Abramsky1999}
\bibitem{Abramsky1999}
\bibinfo{author}{Samson \surnamestart Abramsky\surnameend},
  \bibinfo{author}{Richard \surnamestart Blute\surnameend} \&
  \bibinfo{author}{Prakash \surnamestart Panangaden\surnameend}
  (\bibinfo{year}{1999}): \emph{\bibinfo{title}{Nuclear and trace ideals in
  tensored *-categories}}.
\newblock {\slshape \bibinfo{journal}{Journal of Pure and Applied Algebra}}
  \bibinfo{volume}{143}(\bibinfo{number}{1}), pp. \bibinfo{pages}{3--47},
  \doi{10.1016/S0022-4049(98)00106-6}.
\newblock
  \urlprefix\url{https://www.sciencedirect.com/science/article/pii/S0022404998001066}.

\bibitemdeclare{inproceedings}{Bagnol2015}
\bibitem{Bagnol2015}
\bibinfo{author}{Marc \surnamestart Bagnol\surnameend} (\bibinfo{year}{2015}):
  \emph{\bibinfo{title}{Representation of Partial Traces}}.
\newblock In \bibinfo{editor}{Dan~R. \surnamestart Ghica\surnameend}, editor:
  {\slshape \bibinfo{booktitle}{The 31st Conference on the Mathematical
  Foundations of Programming Semantics, {MFPS} 2015, Nijmegen, The Netherlands,
  June 22-25, 2015}}, {\slshape \bibinfo{series}{Electronic Notes in
  Theoretical Computer Science}} \bibinfo{volume}{319},
  \bibinfo{publisher}{Elsevier}, pp. \bibinfo{pages}{37--49},
  \doi{10.1016/J.ENTCS.2015.12.004}.

\bibitemdeclare{article}{Chiribella2008}
\bibitem{Chiribella2008}
\bibinfo{author}{Giulio \surnamestart Chiribella\surnameend},
  \bibinfo{author}{G~Mauro \surnamestart D’Ariano\surnameend} \&
  \bibinfo{author}{Paolo \surnamestart Perinotti\surnameend}
  (\bibinfo{year}{2008}): \emph{\bibinfo{title}{Quantum circuit architecture}}.
\newblock {\slshape \bibinfo{journal}{Physical review letters}}
  \bibinfo{volume}{101}(\bibinfo{number}{6}), p. \bibinfo{pages}{060401},
  \doi{10.1103/PhysRevLett.101.060401}.

\bibitemdeclare{article}{Cho2019}
\bibitem{Cho2019}
\bibinfo{author}{Kenta \surnamestart Cho\surnameend} \& \bibinfo{author}{Bart
  \surnamestart Jacobs\surnameend} (\bibinfo{year}{2019}):
  \emph{\bibinfo{title}{Disintegration and Bayesian inversion via string
  diagrams}}.
\newblock {\slshape \bibinfo{journal}{Math. Struct. Comput. Sci.}}
  \bibinfo{volume}{29}(\bibinfo{number}{7}), pp. \bibinfo{pages}{938--971},
  \doi{10.1017/S0960129518000488}.

\bibitemdeclare{inproceedings}{Lavore2023}
\bibitem{Lavore2023}
\bibinfo{author}{Elena \surnamestart Di~Lavore\surnameend} \&
  \bibinfo{author}{Mario \surnamestart Román\surnameend}
  (\bibinfo{year}{2023}): \emph{\bibinfo{title}{Evidential Decision Theory via
  Partial Markov Categories}}.
\newblock In: {\slshape \bibinfo{booktitle}{2023 38th Annual ACM/IEEE Symposium
  on Logic in Computer Science (LICS)}}, pp. \bibinfo{pages}{1--14},
  \doi{10.1109/LICS56636.2023.10175776}.

\bibitemdeclare{article}{Fritz2020b}
\bibitem{Fritz2020b}
\bibinfo{author}{Tobias \surnamestart Fritz\surnameend} (\bibinfo{year}{2019}):
  \emph{\bibinfo{title}{A synthetic approach to Markov kernels, conditional
  independence, and theorems on sufficient statistics}}.
\newblock {\slshape \bibinfo{journal}{CoRR}} \bibinfo{volume}{abs/1908.07021},
  \doi{10.48550/arXiv.1908.07021}.
\newblock \eprint{1908.07021}.

\bibitemdeclare{article}{Fritz2023}
\bibitem{Fritz2023}
\bibinfo{author}{Tobias \surnamestart Fritz\surnameend},
  \bibinfo{author}{Tom{\'{a}}s \surnamestart Gonda\surnameend},
  \bibinfo{author}{Nicholas~Gauguin \surnamestart
  Houghton{-}Larsen\surnameend}, \bibinfo{author}{Antonio \surnamestart
  Lorenzin\surnameend}, \bibinfo{author}{Paolo \surnamestart
  Perrone\surnameend} \& \bibinfo{author}{Dario \surnamestart Stein\surnameend}
  (\bibinfo{year}{2023}): \emph{\bibinfo{title}{Dilations and information flow
  axioms in categorical probability}}.
\newblock {\slshape \bibinfo{journal}{Math. Struct. Comput. Sci.}}
  \bibinfo{volume}{33}(\bibinfo{number}{10}), pp. \bibinfo{pages}{913--957},
  \doi{10.1017/S0960129523000324}.

\bibitemdeclare{article}{Fritz2023a}
\bibitem{Fritz2023a}
\bibinfo{author}{Tobias \surnamestart Fritz\surnameend},
  \bibinfo{author}{Tom{\'a}{\v{s}} \surnamestart Gonda\surnameend},
  \bibinfo{author}{Antonio \surnamestart Lorenzin\surnameend},
  \bibinfo{author}{Paolo \surnamestart Perrone\surnameend} \&
  \bibinfo{author}{Dario \surnamestart Stein\surnameend}
  (\bibinfo{year}{2023}): \emph{\bibinfo{title}{Absolute continuity, supports
  and idempotent splitting in categorical probability}}.
\newblock {\slshape \bibinfo{journal}{arXiv preprint arXiv:2308.00651}},
  \doi{10.48550/arXiv.2308.00651}.

\bibitemdeclare{misc}{Fritz2021}
\bibitem{Fritz2021}
\bibinfo{author}{Tobias \surnamestart Fritz\surnameend},
  \bibinfo{author}{Tom{\'a}{\v s} \surnamestart Gonda\surnameend} \&
  \bibinfo{author}{Paolo \surnamestart Perrone\surnameend}
  (\bibinfo{year}{2021}): \emph{\bibinfo{title}{De {F}inetti's Theorem in
  Categorical Probability}}, \doi{https://doi.org/10.48550/arXiv.2105.02639}.
\newblock \eprint{2105.02639}.

\bibitemdeclare{article}{Fritz2020a}
\bibitem{Fritz2020a}
\bibinfo{author}{Tobias \surnamestart Fritz\surnameend},
  \bibinfo{author}{Tom{\'{a}}s \surnamestart Gonda\surnameend},
  \bibinfo{author}{Paolo \surnamestart Perrone\surnameend} \&
  \bibinfo{author}{Eigil~Fjeldgren \surnamestart Rischel\surnameend}
  (\bibinfo{year}{2023}): \emph{\bibinfo{title}{Representable Markov categories
  and comparison of statistical experiments in categorical probability}}.
\newblock {\slshape \bibinfo{journal}{Theor. Comput. Sci.}}
  \bibinfo{volume}{961}, p. \bibinfo{pages}{113896},
  \doi{10.1016/J.TCS.2023.113896}.

\bibitemdeclare{article}{Fritz2023c}
\bibitem{Fritz2023c}
\bibinfo{author}{Tobias \surnamestart Fritz\surnameend} \&
  \bibinfo{author}{Andreas \surnamestart Klingler\surnameend}
  (\bibinfo{year}{2023}): \emph{\bibinfo{title}{The d-Separation Criterion in
  Categorical Probability}}.
\newblock {\slshape \bibinfo{journal}{J. Mach. Learn. Res.}}
  \bibinfo{volume}{24}, pp. \bibinfo{pages}{46:1--46:49},
  \doi{10.48550/arXiv.2207.05740}.
\newblock \urlprefix\url{https://jmlr.org/papers/v24/22-0916.html}.

\bibitemdeclare{article}{Fritz2023b}
\bibitem{Fritz2023b}
\bibinfo{author}{Tobias \surnamestart Fritz\surnameend} \&
  \bibinfo{author}{Wendong \surnamestart Liang\surnameend}
  (\bibinfo{year}{2023}): \emph{\bibinfo{title}{Free gs-Monoidal Categories and
  Free Markov Categories}}.
\newblock {\slshape \bibinfo{journal}{Appl. Categorical Struct.}}
  \bibinfo{volume}{31}(\bibinfo{number}{2}), p.~\bibinfo{pages}{21},
  \doi{10.1007/s10485-023-09717-0}.

\bibitemdeclare{article}{Fritz2020}
\bibitem{Fritz2020}
\bibinfo{author}{Tobias \surnamestart Fritz\surnameend} \&
  \bibinfo{author}{Eigil~Fjeldgren \surnamestart Rischel\surnameend}
  (\bibinfo{year}{2020}): \emph{\bibinfo{title}{Infinite products and zero-one
  laws in categorical probability}}.
\newblock {\slshape \bibinfo{journal}{Compositionality}} \bibinfo{volume}{2},
  p.~\bibinfo{pages}{3}, \doi{10.32408/COMPOSITIONALITY-2-3}.

\bibitemdeclare{inproceedings}{Haghverdi2005}
\bibitem{Haghverdi2005}
\bibinfo{author}{Esfandiar \surnamestart Haghverdi\surnameend} \&
  \bibinfo{author}{Philip~J. \surnamestart Scott\surnameend}
  (\bibinfo{year}{2005}): \emph{\bibinfo{title}{Towards a Typed Geometry of
  Interaction}}.
\newblock In \bibinfo{editor}{C.{-}H.~Luke \surnamestart Ong\surnameend},
  editor: {\slshape \bibinfo{booktitle}{Computer Science Logic, 19th
  International Workshop, {CSL} 2005, 14th Annual Conference of the EACSL,
  Oxford, UK, August 22-25, 2005, Proceedings}}, {\slshape
  \bibinfo{series}{Lecture Notes in Computer Science}} \bibinfo{volume}{3634},
  \bibinfo{publisher}{Springer}, pp. \bibinfo{pages}{216--231},
  \doi{10.1007/11538363\_16}.

\bibitemdeclare{inproceedings}{Hefford2022}
\bibitem{Hefford2022}
\bibinfo{author}{James \surnamestart Hefford\surnameend} \&
  \bibinfo{author}{Cole \surnamestart Comfort\surnameend}
  (\bibinfo{year}{2022}): \emph{\bibinfo{title}{Coend Optics for Quantum
  Combs}}.
\newblock In \bibinfo{editor}{Jade \surnamestart Master\surnameend} \&
  \bibinfo{editor}{Martha \surnamestart Lewis\surnameend}, editors: {\slshape
  \bibinfo{booktitle}{Proceedings Fifth International Conference on Applied
  Category Theory, {ACT} 2022, Glasgow, United Kingdom, 18-22 July 2022}},
  {\slshape \bibinfo{series}{{EPTCS}}} \bibinfo{volume}{380}, pp.
  \bibinfo{pages}{63--76}, \doi{10.4204/EPTCS.380.4}.

\bibitemdeclare{phdthesis}{HoughtonLarsen2021}
\bibitem{HoughtonLarsen2021}
\bibinfo{author}{Nicholas~Gauguin \surnamestart Houghton-Larsen\surnameend}
  (\bibinfo{year}{2021}): \emph{\bibinfo{title}{A Mathematical Framework for
  Causally Structured Dilations and its Relation to Quantum Self-Testing}}.
\newblock Ph.D. thesis, \bibinfo{school}{University of Copenhagen},
  \doi{10.48550/arXiv.2103.02302}.

\bibitemdeclare{article}{Jacobs2020a}
\bibitem{Jacobs2020a}
\bibinfo{author}{B.~\surnamestart Jacobs\surnameend} (\bibinfo{year}{2020}):
  \emph{\bibinfo{title}{A channel-based perspective on conjugate priors}}.
\newblock {\slshape \bibinfo{journal}{Mathematical Structures in Computer
  Science}} \bibinfo{volume}{30}(\bibinfo{number}{1}), p.
  \bibinfo{pages}{44–61}, \doi{10.1017/S0960129519000082}.

\bibitemdeclare{inproceedings}{Jacobs2019}
\bibitem{Jacobs2019}
\bibinfo{author}{Bart \surnamestart Jacobs\surnameend}, \bibinfo{author}{Aleks
  \surnamestart Kissinger\surnameend} \& \bibinfo{author}{Fabio \surnamestart
  Zanasi\surnameend} (\bibinfo{year}{2019}): \emph{\bibinfo{title}{Causal
  Inference by String Diagram Surgery}}.
\newblock In \bibinfo{editor}{Mikolaj \surnamestart Bojanczyk\surnameend} \&
  \bibinfo{editor}{Alex \surnamestart Simpson\surnameend}, editors: {\slshape
  \bibinfo{booktitle}{Foundations of Software Science and Computation
  Structures - 22nd International Conference, {FOSSACS} 2019, Held as Part of
  the European Joint Conferences on Theory and Practice of Software, {ETAPS}
  2019, Prague, Czech Republic, April 6-11, 2019, Proceedings}}, {\slshape
  \bibinfo{series}{Lecture Notes in Computer Science}} \bibinfo{volume}{11425},
  \bibinfo{publisher}{Springer}, pp. \bibinfo{pages}{313--329},
  \doi{10.1007/978-3-030-17127-8\_18}.

\bibitemdeclare{article}{Joyal1996}
\bibitem{Joyal1996}
\bibinfo{author}{André \surnamestart Joyal\surnameend}, \bibinfo{author}{Ross
  \surnamestart Street\surnameend} \& \bibinfo{author}{Dominic \surnamestart
  Verity\surnameend} (\bibinfo{year}{1996}): \emph{\bibinfo{title}{Traced
  monoidal categories}}.
\newblock {\slshape \bibinfo{journal}{Mathematical Proceedings of the Cambridge
  Philosophical Society}} \bibinfo{volume}{119}(\bibinfo{number}{3}), p.
  \bibinfo{pages}{447–468}, \doi{10.1017/S0305004100074338}.

\bibitemdeclare{article}{kissinger2019}
\bibitem{kissinger2019}
\bibinfo{author}{Aleks \surnamestart Kissinger\surnameend} \&
  \bibinfo{author}{Sander \surnamestart Uijlen\surnameend}
  (\bibinfo{year}{2019}): \emph{\bibinfo{title}{A categorical semantics for
  causal structure}}.
\newblock {\slshape \bibinfo{journal}{Logical Methods in Computer Science}}
  \bibinfo{volume}{15}(\bibinfo{number}{3}), \doi{10.23638/LMCS-15(3:15)2019}.

\bibitemdeclare{book}{pearl2009causality}
\bibitem{pearl2009causality}
\bibinfo{author}{J~\surnamestart Pearl\surnameend} (\bibinfo{year}{2009}):
  \emph{\bibinfo{title}{Causality}}.
\newblock \bibinfo{publisher}{Cambridge university press},
  \doi{10.1017/cbo9780511803161}.

\bibitemdeclare{article}{Roman2020}
\bibitem{Roman2020}
\bibinfo{author}{Mario \surnamestart Rom{\'{a}}n\surnameend}
  (\bibinfo{year}{2020}): \emph{\bibinfo{title}{Comb Diagrams for Discrete-Time
  Feedback}}.
\newblock {\slshape \bibinfo{journal}{CoRR}} \bibinfo{volume}{abs/2003.06214},
  \doi{10.48550/arXiv.2003.06214}.
\newblock \eprint{2003.06214}.

\bibitemdeclare{inproceedings}{Roman2020a}
\bibitem{Roman2020a}
\bibinfo{author}{Mario \surnamestart Rom{\'{a}}n\surnameend}
  (\bibinfo{year}{2020}): \emph{\bibinfo{title}{Open Diagrams via Coend
  Calculus}}.
\newblock In \bibinfo{editor}{David~I. \surnamestart Spivak\surnameend} \&
  \bibinfo{editor}{Jamie \surnamestart Vicary\surnameend}, editors: {\slshape
  \bibinfo{booktitle}{Proceedings of the 3rd Annual International Applied
  Category Theory Conference 2020, {ACT} 2020, Cambridge, USA, 6-10th July
  2020}}, {\slshape \bibinfo{series}{{EPTCS}}} \bibinfo{volume}{333}, pp.
  \bibinfo{pages}{65--78}, \doi{10.4204/EPTCS.333.5}.

\bibitemdeclare{inbook}{Selinger2011a}
\bibitem{Selinger2011a}
\bibinfo{author}{P.~\surnamestart Selinger\surnameend} (\bibinfo{year}{2011}):
  \emph{\bibinfo{title}{A Survey of Graphical Languages for Monoidal
  Categories}}, pp. \bibinfo{pages}{289--355}.
\newblock \bibinfo{publisher}{Springer Berlin Heidelberg},
  \bibinfo{address}{Berlin, Heidelberg}, \doi{10.1007/978-3-642-12821-9_4}.

\bibitemdeclare{phdthesis}{Stein2021}
\bibitem{Stein2021}
\bibinfo{author}{Dario \surnamestart Stein\surnameend} (\bibinfo{year}{2021}):
  \emph{\bibinfo{title}{Structural Foundations for Probabilistic Programming
  Languages}}.
\newblock Ph.D. thesis, \bibinfo{school}{University of Oxford}.

\bibitemdeclare{misc}{Stein2021c}
\bibitem{Stein2021c}
\bibinfo{author}{Dario \surnamestart Stein\surnameend} \& \bibinfo{author}{Sam
  \surnamestart Staton\surnameend} (\bibinfo{year}{2021}):
  \emph{\bibinfo{title}{Compositional Semantics for Probabilistic Programs with
  Exact Conditioning (long version)}},
  \doi{https://doi.org/10.48550/arXiv.2101.11351}.
\newblock \eprint{2101.11351}.

\bibitemdeclare{article}{Virgo2023}
\bibitem{Virgo2023}
\bibinfo{author}{Nathaniel \surnamestart Virgo\surnameend}
  (\bibinfo{year}{2023}): \emph{\bibinfo{title}{Unifilar Machines and the
  Adjoint Structure of Bayesian Models}}.
\newblock {\slshape \bibinfo{journal}{arXiv preprint arXiv:2305.02826}},
  \doi{10.48550/arXiv.2305.02826}.

\end{thebibliography}

\section{Appendix}\label{sec:appendix}

\begin{proof}[Proof of Proposition~\ref{prop:trace-discardable}]
	The trace of a discardable morphism $f$ need not itself be discardable. For example, in the compact closed category $\mat(\R^+)$, the trace of $\id_W : W \to W$ is the scalar $|W| : I \to I$ which is not normalized (i.e. equal to $1$). On the other hand, if $f$ is non-signalling, we obtain 
	\[ \input{discarding.tikz} \]
\end{proof}

\subsection{Appendix to Section~\ref{sec:atomicmc}}\label{app:atomicmc}

\begin{proof}[Proof of Proposition~\ref{prop:atomicmorphisms}]
The statement for morphisms of full support is immediate. For a deterministic morphism $p$, we reason straightforwardly
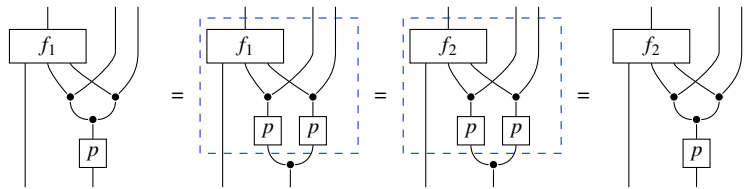
\[ \input{det-atomic.tikz} \]
For the third point, let $g$ be deterministic and $p$ atomic, and assume
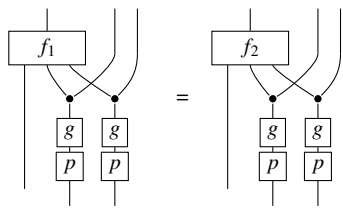
\[ \input{det-atomic-comp.tikz} \]
By causality \cite[A.14]{Fritz2023}, we may strengthen this equation and marginalize to obtain
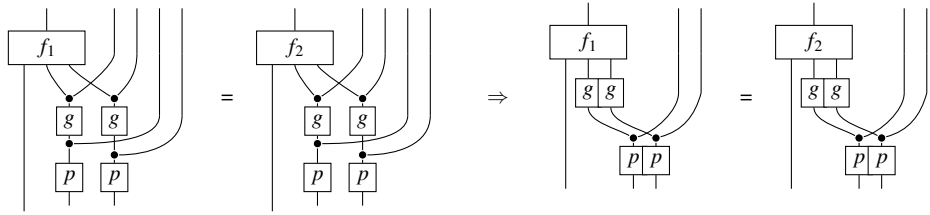
\[ \input{det-atomic-comp-2.tikz} \]
Now, we use atomicity of $p$, postcompose the two rightmost wires with $g$ and use determinism to simplify
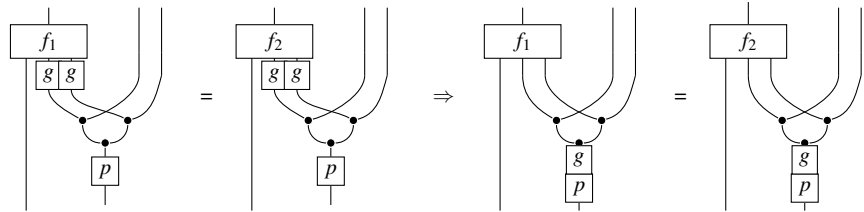
\[ \input{det-atomic-comp-3.tikz} \]
\end{proof}

\begin{example}\label{ex:atomic-no-composition}
Atomic morphisms need not be closed under composition: In $\borelstoch$, let $X=[0,1]$ and define the morphism $p : X \otimes 2 \to X$ by
\[ p(x,c) = \begin{cases}
	\delta_x, &c=0 \\
	\nu, &c=1
\end{cases}\]
Then the atoms of $p$ are all of $X$, i.e. $p$ is atomic. However, for all deterministic states $\sigma : I \to X \otimes 2$ of the form $\sigma = \delta_{(x,1)}$, we have that $p \circ \sigma = \nu$ is not atomic. 
\end{example}

\begin{proof}[Proof of Lemma~\ref{lem:disint-indep}]
	First, by marginalizing $Y$ and cancellability of $W$, we conclude that $f_1 = f_2$ and write $f$ indiscriminately. Now we show that the morphisms $c_i = \del_X \otimes g_i$ (for $i=1,2$) and are $\phi \otimes \phi$-almost surely equal, where $\phi = (\id_X \otimes f) \circ \Delta_X$. 
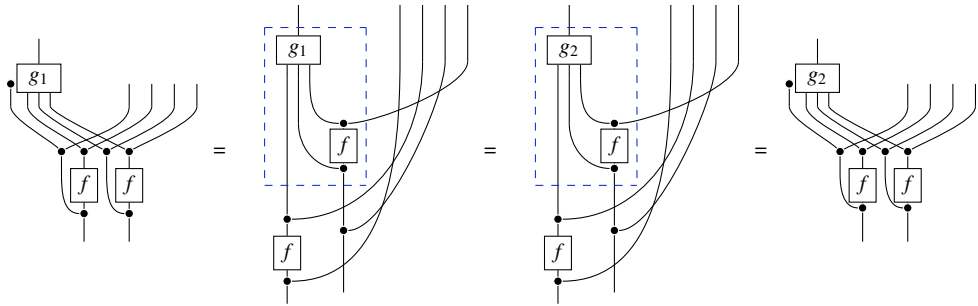
	\[ \input{lem-pp-as.tikz} \]
	Applying atomicity to $\phi$, we obtain that that $c_1, c_2$ are also $(\Delta \circ \phi)$-almost surely equal, i.e.
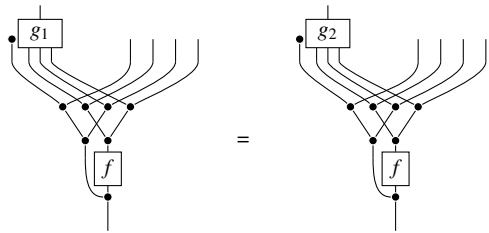
	\[ \input{lem-cp-as.tikz} \]
	From this, we obtain the desired conclusion by marginalizing the four wires on the right.
\end{proof}

\begin{proof}[Proof of Lemma~\ref{lem:atomic-contr}]
	For $i=1,2$, disintegrate $f_i$ as 
	\[ \input{f-disint.tikz}\]
	Then
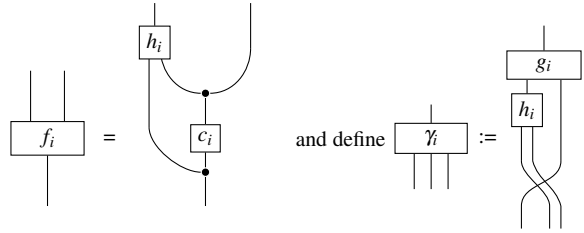
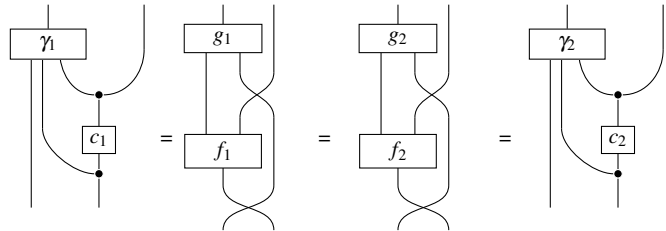
	\[ \input{f-disint-assumption.tikz}\]
	So by Lemma~\ref{lem:disint-indep}, we conclude
	\[ \input{f-disint-conclusion.tikz}\]
\end{proof}

\subsection{Appendix to Section~\ref{sec:causaltraces-laws}}\label{app:traceproofs}

\begin{proof}[Proof of Tightening (Proposition~\ref{prop:tightening})]
	Write $k$ for the composite $(h \otimes \id_W) \circ f \circ (g \otimes \id_W)$; it is easy to see that $k$ is traceable. Using the Bayesian inverse $(f_s)^\dagger_g$, we obtain the following disintegration
	\[ \input{tightening-fact.tikz} \]
	Using that disintegration, we compute as desired
	\[ \input{tightening-proof.tikz} \]
\end{proof}

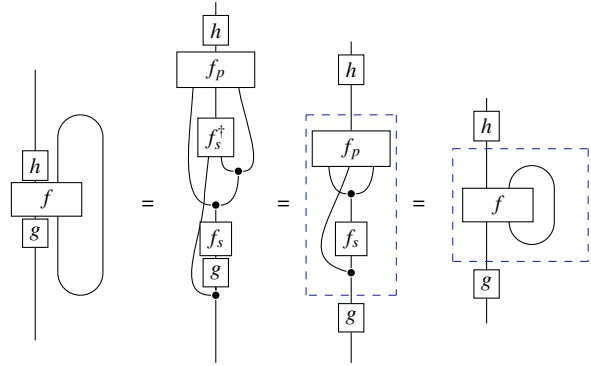
\begin{proof}[Proof of Sliding (Proposition~\ref{prop:sliding})]
	Traceability is straightforward. Using the Bayesian inverse $g^\dagger_{f_s}$, we establish the following disintegrations:
	\[ \input{sliding-disint.tikz} \]
	Using these disintegrations, we show
	\[ \input{sliding-proof.tikz} \]
\end{proof}

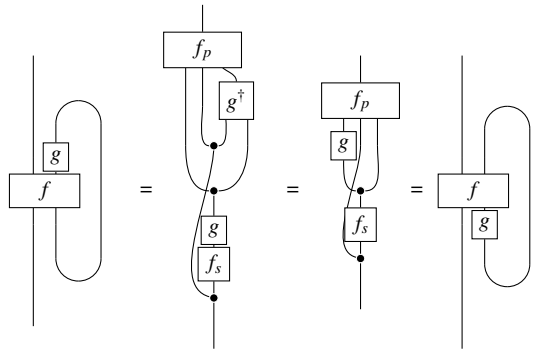
\begin{proof}[Proof of Associativity (Proposition~\ref{prop:associativity})]	
Assume that $f : X \otimes (U \otimes V) \to Y \otimes (U \otimes V)$ is non-signalling in $U \otimes V$; that is
	\[ \input{coherence-nonsig.tikz} \]
	hence $f$ is in particular non-signalling from $V$ to $V$. Choose a disintegration $f_p, f_s$, and condition $f_s$ further, to obtain
	\[ \input{coherence-disint.tikz} \]
	Then we can give the following disintegration with respect to $V$
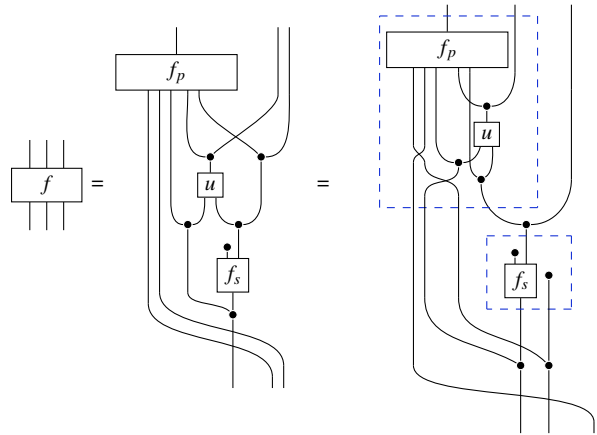
	\[ \input{coherence-disint-2.tikz} \]
	Using this disintegration, we obtain $\trace V(f)$ as
	\[ \input{coherence-proof-1.tikz} \]
	Of this, we can in turn compute the trace
	\[ \input{coherence-proof-2.tikz} \]
\end{proof}

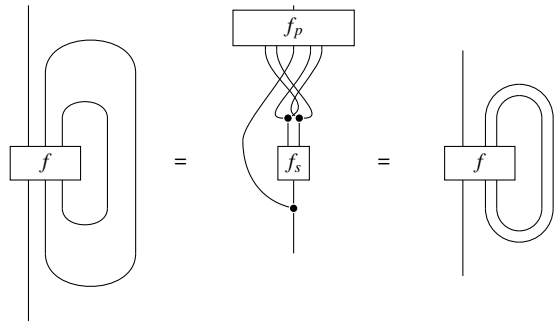
\begin{proof}[Proof of Superposition (Proposition~\ref{prop:superposition})]
	We make use of the following disintegration
	\[ \input{strength-disint.tikz} \]
	to obtain
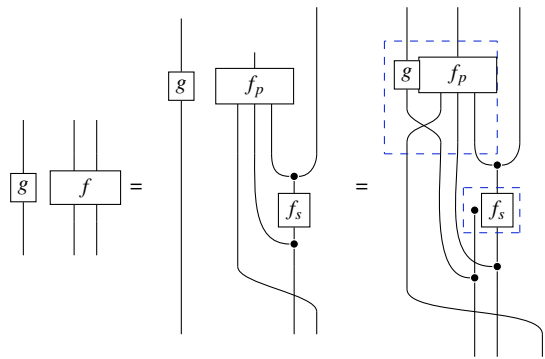
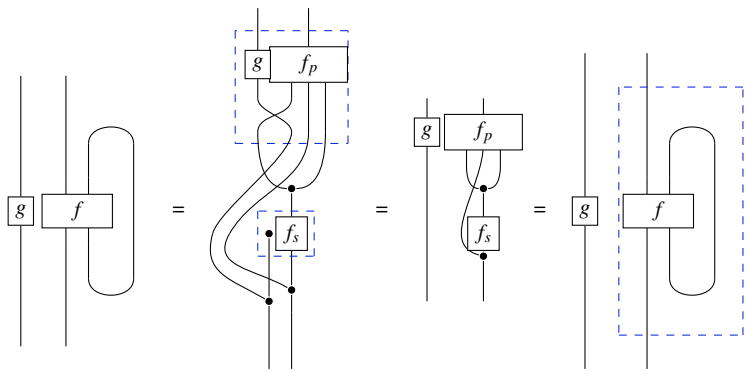
	\[ \input{strength-proof.tikz} \]
\end{proof}

\begin{proof}[Proof of Yanking (Proposition~\ref{prop:yanking})]
	The symmetry has the following disintegration, hence
	\[ \input{symmetry-disint.tikz} \]
\end{proof}

\subsection{Appendix to Section \ref{secfreemarkovtrace}} \label{app:freemarkov}

We recall the construction of free Markov categories \cite{Fritz2023b}. 
\begin{definition}[\cite{Fritz2023b}]
	The category $\cat I$ is defined as follows.
	\begin{itemize}
		\item Objects are pairs of natural numbers $(m,n) \in \N \times \N$, and there is an extra object $\ast$.
		\item The only non-identity morphisms are $\hgin_1,\cdots,\hgin_m, \hgout_1,\cdots, \hgout_n: (m,n) \to \ast$.
	\end{itemize}
	A hypergraph is a functor $\cat I \to \sets$. Hypergraphs form a functor category $\hyp$. 
\end{definition}
We also use the following notation for a hypergraph $G : \cat I \to \sets$.
\begin{itemize}
	\item $W(G) = G(\ast)$ is the set of wires.
	\item $B_{m,n}(G)$ is the set of boxes with $m$ inputs and $n$ outputs. $B(G) = \bigsqcup_{m,n \in \N} B_{m,n}(G)$ is the set of all boxes. \item We abbreviate $G(\hgin_i)$ to $\hgin_i$ and $G(\hgout_i)$ to $\hgout_i$. These assign the $i$th input/output wire to each box.
	\item For $b \in B_{m,n}(G)$, $w \in W(G)$ 
	\begin{align*}
		\hgin(b,w) &= \{\hgin_i(b) : i \in \{1,\cdots,,m\}\} \\
		\hgout(b,w) &= \{\hgout_i(b) : i \in \{1,\cdots,n\}\}	
	\end{align*}
	These numbers tell how many times a wire is the input/output of a box.
\end{itemize}

\begin{definition}[\cite{Fritz2023b}]
	A hypergraph is finite if the set of wires $W(G)$ and the set of boxes $B(G)$ is finite. We denote the subcategory of finite hypergraphs by $\finhyp$.
\end{definition}

Every monoidal signature $\Sigma$ gives rise to a finite hypergraph which we also denote by $\Sigma$. A labelling of wires and boxes in a hypergraph $G$ is a hypergraph homomorphism (natural transformation) to $\Sigma$.

\begin{definition}[\cite{Fritz2023b}]
	Markov string diagrams over the monoidal signature $\Sigma$ are (isomorphism classes of) cospans in the slice category $\finhyp/\Sigma$, that is of the form
	\begin{center}
		\begin{tikzcd}
			\underline{m} \arrow[r,"p"] \arrow[dr] & G  \arrow[d] & \underline{n} \arrow[l,"q",swap]  \arrow[dl] \\
			& \Sigma 		
		\end{tikzcd}
	\end{center}
	in $\finhyp$, such that
	\begin{enumerate}
		\item $\underline{m}$ and $\underline{n}$ are discrete hypergraphs ($B(\underline{m}) = B(\underline{n}) =\emptyset$).
		\item G is acyclic, i.e. it contains no cycles. A path is a finite, alternating sequence of wires and boxes $(w_1, b_1, \cdots w_n, b_n, w_{n+1})$ such that $\hgin(b_i,w_i) > 0$, and $\hgout(b_i,w_{i+1} > 0)$ for all $1 \le i \le n$. A cycle is a path that additionally satisfies $w_1 = w_{n+1}$.	
		\item The cospan satisfies left monogamy: for all wires $w \in W(G)$
		\[ |p^{-1}(w)| + \sum_{b \in B(G)} \hgout(b,w) = 1 \]
		That is, every wire is either connected to exactly one input port or is the output of a single box.
		\item There are no eliminable boxes. A box is eliminable if none of its output wires are connected to an output port or to the input of a box. That is, an eliminable box $b$ satisfies for all $w \in W(G)$
		\[ \hgout(b,w) > 0 \Longrightarrow q^{-1}(w) = 0 \wedge \forall b' \in B(G).\hgin(b',w) = 0	 \]
	\end{enumerate}
\end{definition}
\noindent
We leave the labelling implicit and write the cospan as $\underline{m} \xrightarrow{p} G \xleftarrow{q} \underline{n}$.

\begin{proposition}[\cite{Fritz2023b}]
	The free Markov category $\freemarkov_\Sigma$ over a monoidal signature $\Sigma$ can be constructed as follows:
	\begin{enumerate}
		\item Objects are homomorphisms $\underline{m} \to \Sigma$, that is an $m$-long list of types in $\Sigma$.
		\item Morphisms are (isomorphism classes of) Markov string diagrams, which compose by pushout and subsequent normalisation (elimination of eliminable boxes).
		\item The tensor is given by coproduct.
		\item Copy and delete are represented by the cospans
		\[ \input{copydel-freemarkov.tikz}\]
	\end{enumerate}
\end{proposition}
To ease notation we write $f : n \to m$ for a morphism $f : (\underline{m} \to \Sigma) \to (\underline{n} \to \Sigma)$, thus leaving the labelling implicit.

\subsection{Appendix to Section~\ref{sec:combs}}\label{app:combs}

\begin{proof}[Proof of Theorem~\ref{thm:combs-nonsignalling}]
For each comb $\comb f g : (A,A') \to (B,B')$, its extension is a non-signalling morphism $A \otimes B' \to A' \otimes B$. Conversely, from a non-signalling morphism $f : A \otimes B' \to A' \otimes B$, we construct a disintegration
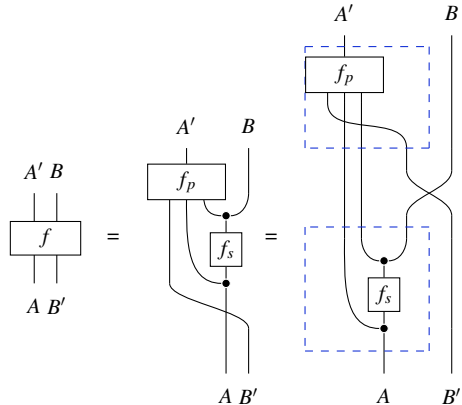
\[ \input{mc-comb-disint.tikz} \]
which defines a comb with environment $E=A \otimes B$ whose extension is $f$. By Proposition~\ref{prop:extensional-implies-contextual}, the resulting comb is unique up to contextual equivalence.
\end{proof}

\paragraph{Universal Dilations}

In the terminology of \cite[Definition~2.4.1]{HoughtonLarsen2021}, a \emph{universal dilation} of a morphism $p : X \to Y$ is another morphism $\phi : X \to E \otimes Y$ satisfying
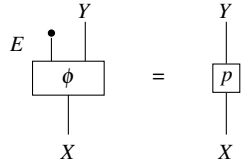
\[ \input{univ-dil-1.tikz} \]
such that for all $\Pi : X \otimes W' \to W \otimes Y$  
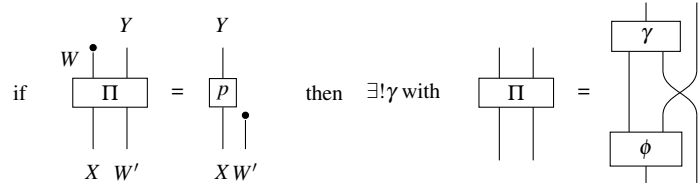
\[ \input{univ-dil-2.tikz} \]

\begin{proof}[Proof of Proposition~\ref{prop:univ-dil-optic}]
Let $\combe {E_1} {f_1} {g_1}, \combe {E_2} {f_2} {g_2} : (A,A') \to (B,B')$ be two combs which are extensionally equivalent. Define using cancellability $p : A \to B$ as the common morphism
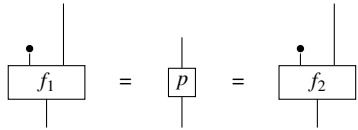
\[ \input{univ-dil-3.tikz} \]
and choose a universal dilation $\phi : A \to E \otimes B$ of $p$. Using universality, the $f_i$ must factor through $\phi$ as
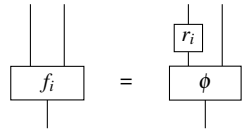
\[ \input{univ-dil-4.tikz} \]
By uniqueness of factorization, we have
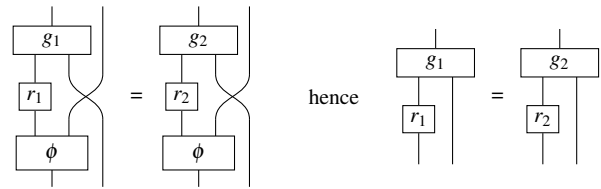
\[ \input{univ-dil-5.tikz} \]
Now we reason modulo optic equivalence that
\[ \input{univ-dil-6.tikz} \]
\end{proof}

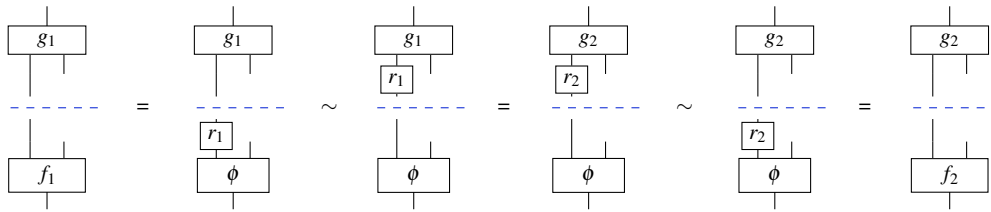
\end{document}